\documentclass{amsart}
\usepackage{amsfonts}
\usepackage{amsmath,amscd}
\usepackage{amssymb}
\usepackage{amsthm}
\usepackage{newlfont}
\newcommand{\f}{\frac}

\newcommand{\ds}{\displaystyle}

 \newtheorem{thm}{Theorem}[section]
 
 \newtheorem{lem}[thm]{Lemma}
 
 \theoremstyle{definition}
 
 \theoremstyle{remark}

 \numberwithin{equation}{section}

\begin{document}

\title[Classification of $p$-groups via their $2$-nilpotent multipliers]
 {Classification of $p$-groups via their $2$-nilpotent multipliers}

\author[P. Niroomand]{Peyman Niroomand}
\address{School of Mathematics and Computer Science\\
Damghan University, Damghan, Iran}
\email{niroomand@du.ac.ir, p\underline~ niroomand@yahoo.com}
\author[Mohsen Parvizi]{Mohsen Parvizi}
\address{Department of Pure Mathematics, Ferdowsi University of Mashhad, Mashhad, Iran.}
\email{parvizi@math.um.ac.ir}






\begin{abstract}
For a $p$-group of order $p^n$, it is known that  the order of $2$-nilpotent multiplier is equal to $|\mathcal{M}^{(2)}(G)|=p^{\f12n(n-1)(n-2)+3-s_2(G)}$
for an integer $s_2(G)$. In this article, we  characterize all of non abelian $p$-groups satisfying in
$s_2(G)\in\{1,2,3\}.$ \end{abstract}

\maketitle
\section{Preliminaries}
The $2$-nilpotent multiplier of a $p$-group is a particular of the group $\mathcal{M}^{(c)}(G)$, which is Baer invariant of $G$,  introduced by Baer
in \cite{ba}, and may be found also in \cite{fr, ma}.
When $c=1$,  $\mathcal{M}^{(1)}(G)=\mathcal{M}(G)$ is well-known, and it is called the Schur multiplier of $G$.

The motivation survey of the $2$-nilpotent multiplier of $G$ comes from \cite{el2} that is connection to isologism of groups which is an
important instrument for classifying groups such as $p$-groups.

Recall from \cite{hal} a group as form $G\cong H/Z_2(H)$ is called $2$-capable. Choose a free presentation $G\cong F/R$, and consider the natural epimorphism $\alpha:F/[R,F,f]\rightarrow G$. We may define $Z_2^{*}(G)=\alpha(Z_2(F/[R,F,F]))$. \cite[Proposition 1.2]{el2} gives an instrument tools
to verify when $G$ is $2$-capable. More precisely, $G$ is $2$-capable if and only if $Z_2^{*}(G)=1$.

To be more precise on the order $\mathcal{M}^{(2)}(G)$ and the fact that the natural epimorphism $\mathcal{M}^{(2)}(G)\rightarrow \mathcal{M}^{(2)}(G/N)$
is a monomorphism for a normal subgroup $N$ of $Z_2^{*}(G)$ $($see \cite[Lemma 2.1]{el2}$)$ allow us to characterize which  $p$-group are $2$-capable.
For instance looking \cite{ni3} shows which of extra-special $p$-groups are $2$-capable.

A famous result of Green shows that for a given $p$-group of order $p^n$, $|\mathcal{M}(G)|=p^{\f12n(n-1)-t(G)}$ for $t(G)\geq 0$.
Several authors in \cite{be, zh, ni2, ni3} trying to classify the structure of $G$  in term of $t(G)$ up to $5$.
Restricting on non abelian groups in \cite{ni}, the Green's bound is improved to  $p^{\f12(n-1)(n-2)+1}$, and hence
 there is an integer $s(G)$ such that $|\mathcal{M}(G)|=p^{\f12(n-1)(n-2)+1-s(G)}$. The same result is proved in \cite{ni3} for  $2$-nilpotent multiplier of a group $G$. For a non abelian $p$-group of order $p^n$ there exists a integer $s_2(G)$ such that
 $|\mathcal{M}^{(2)}(G)|=p^{\f12n(n-1)(n-2)+3-s_2(G)}$, and the structure of all $p$-groups are classified when $s_2(G)=0$. By the same motivation in \cite{be, zh, ni2, ni3}, we are interesting
 to characterize $p$-groups up to isomorphisms when $s_2(G)\in\{1,2,3\}$.

Some theorems and lemmas are needed for the next investigation.
\begin{lem}\label{p1} Let $G$ be a finite group and $B\unlhd G$. Set $A = G/B$
\begin{itemize}
\item[(i)] \cite[Proposition 2]{el}

$(a)$. If $B\subseteq Z_2(G)$, then $|\mathcal{M}^{(2)}(G)||B\cap \gamma _{3}(G)|$ divides $|\mathcal{M}^{(2)}(A)|�|\ds(B\otimes \f{G}{\gamma_{3}(G)})\otimes
\f{G}{\gamma_{3}(G)}|�$,

$(b)$. \cite{lu} The sequence

 $(B\wedge G)\wedge G \rightarrow \mathcal{M}^{(2)}(G)\rightarrow\mathcal{M}^{(2)}(G/B)\rightarrow B\cap\gamma_3(G)\rightarrow 1$ is exact,
\item[(ii)]\cite{r1}
$|\mathcal{M}^{(2)}(A)|$ divides $|\mathcal{M}^{(2)}(G)| B\cap \gamma_{3}(G)|/|[[B,G],G]|$.
\end{itemize}
\end{lem}

The following theorem plays an essential role in the rest.

\begin{lem}\label{p2}\cite{r3}
Let $G$ be a finite group. Put $G^{ab}=G/G^{'}.$ Then there is a natural isomorphism
\[\mathcal{M}^{(2)}(G\times H)\cong
\mathcal{M}^{(2)}(G)\times\mathcal{M}^{(2)}(H)\times (G^{ab}\otimes G^{ab})
 \otimes H^{ab}\times (H^{ab}\otimes H^{ab})\otimes G^{ab}.\]
\end{lem}

\begin{lem}\label{p3}
Let $G$ be an extra-special $p$-group of order $p^{2n+1}$.
\begin{itemize}
\item[(i)] If $n>1$,then $\mathcal{M}^{(2)}(G)$ is an elementary abelian $p$-group of order $p^{\frac13(8n^3-2n)}$.
\item[(ii)] Suppose that $|G|=p^3$ and $p$ is odd. Then $\mathcal{M}^{(2)}(G)=\mathbb{Z}_p^{(5)}$ if $G$ is of exponent $p$ and $\mathcal{M}^{(2)}(G)=\mathbb{Z}_p\times \mathbb{Z}_p$ if $G$ is of exponent $p^2$.
\item[(iii)] The quaternion group of order 8 has Klein four-group as the 2-nilpotent multiplier, whereas the 2-nilpotent multiplier of the dihedral group of order 8 is $\mathbb{Z}_2\oplus \mathbb{Z}_4$.
\end{itemize}
\end{lem}

\begin{lem}\label{p4} Let $G=\mathbb{Z}_{p^{m_1}}\oplus \mathbb{Z}_{p^{m_2}}\oplus\cdots\oplus
\mathbb{Z}_{p^{m_k}}$, where $m_1\geq m_2\geq \cdots\geq m_k $ and $\ds\Sigma_{i=1}^km_i=n$. Then
\begin{itemize}
\item[(i)] $|\mathcal{M}^{(2)}(G)|=p^{\f13n(n-1)(n+1)}$ if and only if~ $m_i=1$ for all $i$.
\item[(ii)] $|\mathcal{M}^{(2)}(G)|\leq p^{\f13n(n-1)(n-2)}$ if and only if~ $m_1\geq 2$.
\end{itemize}
\end{lem}

\section{Main Results}
We know that the order of the $2$-nilpotent multiplier of a finite non abelian $p$-group of order $p^n$ is bounded by $p^{\frac13n(n-1)(n-2)+3}$, therefore for any group $G$ there exists a non negative integer $s_{_2}(G)$ for which $|\mathcal{M}^{(2)}(G)|=p^{\frac13n(n-1)(n-2)+3-s_{_2}(G)}$. In this paper, we characterize the explicit structures of finite non abelian $p$-groups when $s_{_2}(G)\in \{1,2,3\}$. It must be noted that the group with $s_{_2}(G)=0$ are completely determined in \cite{ni3} as follows.

\begin{thm}
Let $G$ be a non abelian finite $p$-group of order $p^n$ whose derived subgroup is of order $p$. Then $|\mathcal{M}^{(2)}(G)|\leq
p^{\f13n(n-1)(n-2)+3}$, and the
equality holds if and only if $G\cong E_1\times \mathbb{Z}_p^{(n-3)}$.
\end{thm}

First we state the following theorem from \cite{ni3} to prove the only groups which may have the desired property are those with small derived subgroups.

\begin{thm}\label{1} Let $G$ be a $p$-group of order $p^n$ with $|G^{'}|=p^{m} (m\geq 1)$. Then
\[|\mathcal{M}^{(2)}(G)|\leq p^{\f13(n-m)\big((n+2m-2)(n-m-1)+3(m-1)\big)+3}~ .\]
\end{thm}

Using the above theorem we have

\begin{lem}\label{m>=3}
Let $G$ be a non abelian $p$-group of order $p^n$ with $|G'|\geq p^3$. Then $|\mathcal{M}^{(2)}(G)|\leq
p^{\f13n(n-1)(n-2)-2}$
\end{lem}

\begin{proof}
Just use Theorem \ref{1} and the fact that $n$ is at least 5.
\end{proof}

The following lemma has a completely similar proof to that of Lemma \ref{m>=3}.

\begin{lem}\label{m=2}
Let $G$ be a non abelian $p$-group of order $p^n$ with $|G'|=p^2$. Then $|\mathcal{M}^{(2)}(G)|\leq
p^{\f13n(n-1)(n-2)+1}$
\end{lem}

The following theorem gives an upper bound for the 2-nilpotent multiplier of $G$.

\begin{thm}\label{our}
Let $G$ be a $p$-group and $B$ be a cyclic central subgroup of $G$. Then

\[|\mathcal{M}^{(2)}(G)|\leq |\mathcal{M}^{(2)}\big(G/B\big)||B\otimes G/B\otimes G/G'|.\]
\end{thm}

\begin{proof}
Let $G=F/R$ and $B=S/R$ be free presentations for $G$ and $B$, respectively. Since $B$ ic central, we have $[S,F]\subseteq R$, and also $R\cap S'=[R,S]$ because $B$  is cyclic. Now $S'\subseteq R$, and so $S'=[R,S]$.

 From the definition  \[\mathcal{M}^{(2)}(G)\cong \frac{R\cap \gamma_3(F)}{[R,F,F]}~\text{and}~\mathcal{M}^{(2)}\big(G/B\big)\cong \frac{S\cap \gamma_3(F)}{[S,F,F]},\] and so \[|\mathcal{M}^{(2)}(G)|=|\mathcal{M}^{(2)}\big(G/B\big)||\frac{[S,F,F]}{[R,F,F]}|.\]
The proof is completed if  there exists a well-defined epimorphism \[\psi:S/R\otimes F/S \otimes F/RF'\longrightarrow \frac{[S,F,F]}{[R,F,F]}.\]
Define the map $\psi$ by the rule $\psi(sR,f_1S,f_2RF')=[s,f_1,f_2][R,F,F]$.  It is enough to show that  $\psi$ is a well-defined multi linear mapping. For this, first we show that \[[sr,f_1s',f_2r'\gamma]\equiv [s,f_1,f_2] \pmod{[R,F,F]}\] where $r,r'\in R$, $s,s'\in S$ and $\gamma\in F'$.

Expanding the commutator on the left hand side we have $[sr,f_1s',f_2r'\gamma]=[sr,f_1s',r'\gamma][sr,f_1s',f_2][sr,f_1s',f_2,r'\gamma]$. Trivially $[sr,f_1s',f_2,r'\gamma]\in [S,F,F,F]$, but $[S,F]\subseteq R$ and so $[S,F,F,F]\subseteq [R,F,F]$. On the other hand $[sr,f_1s',r'\gamma]=[sr,f_1s',\gamma][sr,f_1s',r'][sr,f_1s',r',\gamma]$ which is contained in $[S,F,F'][S,F,R]$. A simple use of Three Subgroup Lemma shows the latter is contained in $[R,F,F]$. We proved that $[sr,f_1s',f_2r'\gamma]\equiv[sr,f_1s',f_2] \pmod{[R,F,F]}$. Again using commutator calculus, we get \[[sr,f_1s',f_2]=[sr,s',f_2][sr,s',f_2,[sr,f_1]^{s'}][[sr,f_1]^{s'},f_2].\] It is easy to see that \newline \[[sr,s',f_2][sr,s',f_2,[sr,f_1]^{s'}]\in [S,S,F]=[S',F]=[R,S,F]\subseteq[R,F,F].\] Finally $[[sr,f_1]^{s'},f_2]=[sr,f_1,f_2][sr,f_1,f_2,[sr,f_1,s']][sr,f_1,s',f_2]$, for the last two ones we have $[sr,f_1,f_2,[sr,f_1,s']][sr,f_1,s',f_2]\in [S,F,F,F]\subseteq [R,F,F]$. The first one can be decomposed as \[[sr,f_1,f_2]=[s,f_1,f_2][s,f_1,f_2,[s,f_1,r]][s,f_1,r,f_2][[s,f_1]^r,f_2,[r,f1]][r,f_1,f_2],\] and we have \[\begin{array}{lcl}[s,f_1,f_2,[s,f_1,r]][s,f_1,r,f_2][[s,f_1]^r,f_2,[r,f1]][r,f_1,f_2]
&\in& [S,F,F,F][R,F,F]\vspace{.3cm}\\&\subseteq& [R,F,F].\end{array}\]  The multi linearity of this mapping is an straightforward commutator calculus.
\end{proof}

Considering Lemmas \ref{m>=3} and \ref{m=2} for characterizing all $p$-groups wit $s_2(G)\in \{1,2,3\}$,  it is enough to work with $p$-groups with $|G'|\leq p^2$. First we deal with those groups having commutator subgroup of order $p$. If $G/G'$ is not elementary abelian we have

\begin{lem}\label{nl}
Let $G$ be a $p$-group of order $p^n$ with $G'$ of order $p$. If $G/G'$ is not elementary abelian then
\[|\mathcal{M}^{(2)}(G)|\leq p^{\frac13n(n-1)(n-2)-2}.\]
\end{lem}

\begin{proof}
We use Theorem \ref{our} with $B=G'$ to have \[|\mathcal{M}^{(2)}(G)|\leq |\mathcal{M}^{(2)}\big(G/G'\big)||G'\otimes G/G'\otimes G/G'|.\] Since $G/G'$ is not elementary abelian, by using Lemma \ref{p4} \[|\mathcal{M}^{(2)}\big(G/G'\big)|\leq p^{\frac13(n-1)(n-2)(n-3)}\] Since $|G'\otimes G/G'\otimes G/G'|\leq p^{(n-2)^2}$, an straightforward computation shows the result provided that for $n\geq 4$.
\end{proof}

Now we may assume that $G/G'$ is elementary abelian. In \cite[Lemma 2.1]{ni} $p$-groups with $G'=\phi(G)$ (the Frattini subgroup) of order $p$ are classified as the central product of an extra special $p$-group $H$ with the center of $G$, $Z(G)$. That is $G=H\cdot Z(G)$. Now depending on how $G'$ embeds into $Z(G)$, we have the following lemma which has an straightforward proof.

\begin{lem}
Let $G$ be a $p$-group with $G'=\phi(G)$ of order $p$. Then
\begin{itemize}
\item[(i)] If $G'$ is a direct summand of $Z(G)$ then $G=H\times K$ for some finite abelian group $K$.
\item[(ii)] If $G'$ is not a direct summand of $Z(G)$ then $G=\big(H\cdot \mathbb{Z}_{p^t}\big) \times K$ where $t\geq 2$ and $K$ is a finite abelian $p$-group.
\end{itemize}

\end{lem}

As we consider the groups for which $G/G'$ is elementary abelian, we have only the following two cases. $T$ is an elementary abelian $p$-group.

\begin{itemize}
\item[(i)] $G=H\times T$.
\item[(ii)] $G=H\cdot \mathbb{Z}_{p^2}\times T$.
\end{itemize}

For the groups of type $(i)$ we have.

\begin{thm}\label{2.8}
Let $G=H\times T$ where $H$ is an extra special $p$-group and $T$ is an elementary abelian $p$-group. Then
\begin{itemize}
\item[(i)] if $H=E_1$ then $|\mathcal{M}^{(2)}(G)|=p^{\frac13n(n-1)(n-2)+3}$
\item[(ii)] if $H=D_8$ then $|\mathcal{M}^{(2)}(G)|=2^{\frac13n(n-1)(n-2)+1}$
\item[(iii)] in other cases $|\mathcal{M}^{(2)}(G)|=p^{\frac13n(n-1)(n-2)}$
\end{itemize}
\end{thm}

\begin{proof}
Its just an straightforward computations using Lemmas \ref{p2} and \ref{p3}.
\end{proof}

For the second type,  first we compute the order of the $2$-nilpotent multiplier of $H\cdot \mathbb{Z}_{p^2}$.

\begin{thm}\label{2.9}
Let $G=H\cdot\mathbb{Z}_{p^2}$ be of order $p^n$. Then $|\mathcal{M}^{(2)}(G)|=p^{\frac13n(n-1)(n-2)}$.
\end{thm}

\begin{proof}
Using Theorem \ref{our} with $B=\mathbb{Z}_{p^2}$ we get \[|\mathcal{M}^{(2)}(G)|\leq |\mathcal{M}^{(2)}\big(G/\mathbb{Z}_{p^2}\big)||\mathbb{Z}_{p^2}\otimes G/\mathbb{Z}_{p^2}\otimes G/G'|.\] By assumption $|H|=p^{2m+1}$, we have \[|\mathcal{M}^{(2)}\big(G/\mathbb{Z}_{p^2}\big)|=p^{\frac132m(2m+1)(2m-1)}~ \text{and} ~|\mathbb{Z}_{p^2}\otimes G/\mathbb{Z}_{p^2}\otimes G/G'|=p^{(2m+1)^2}.\] After some computations $|\mathcal{M}^{(2)}(G)|\leq p^{\frac13n(n-1)(n-2)}$. Now Lemma \ref{p1} (a) with $B=G'$ shows $ |\mathcal{M}^{(2)}\big(G/G'\big)|\leq|\mathcal{M}^{(2)}(G)|$. The result now follows by using Lemma \ref{p4}.
\end{proof}

Now the following theorem which has a proof completely similar to the last two ones, completes the groups by type $(ii)$.

\begin{thm}\label{cs}
Let $G=H\cdot\mathbb{Z}_{p^2}\times T$ be of order $p^n$ in which $T$ is an elementary abelian $p$-group and $H$ is an extra-special $p$-groups. Then $|\mathcal{M}^{(2)}(G)|=p^{\frac13n(n-1)(n-2)}$.
\end{thm}

In the rest we concentrate on the groups with the derived subgroup of order $p^2$.

\begin{lem}
Let $G$ be a $p$-group of order $p^n$ with $G'$ of order $p^2$. If $Z(G)$ is not elementary abelian then $|\mathcal{M}^{(2)}(G)|\leq p^{\frac13n(n-1)(n-2)-2}.$
\end{lem}

\begin{proof}
Choose $B\subseteq Z(G)$ cyclic of order $p^2$ and use Theorem \ref{our} to obtain \[|\mathcal{M}^{(2)}(G)|\leq |\mathcal{M}^{(2)}\big(G/B\big)||B\otimes G/B\otimes G/G'|.\] Since \[|\mathcal{M}^{(2)}\big(G/B\big)|\leq p^{\frac13(n-1)(n-2)(n-3)} ~\text{and} ~|B\otimes G/B\otimes G/G'|\leq p^{(n-2)^2},\]  $|\mathcal{M}^{(2)}(G)|\leq p^{\frac13n(n-1)(n-2)-2},$
the result follows.\end{proof}

In the class of groups with an elementary abelian center we must consider the following two cases.

\begin{lem}\label{lw}
Let $G$ be a $p$-group of order $p^n$ with $G'$ of order $p^2$. Let $Z(G)$ be elementary abelian. If $|Z(G)|\geq p^3$ or $|Z(G)|=p^2$ and $G'\neq Z(G)$ then $|\mathcal{M}^{(2)}(G)|\leq p^{\frac13n(n-1)(n-2)-2}.$
\end{lem}

\begin{proof}
There exists a central subgroup $K$ of order $p$ with $K\cap G'=1$. Using Lemma \ref{p1} (a), we have $|\mathcal{M}^{(2)}(G)|\leq |\mathcal{M}^{(2)}\big(G/K\big)||K\otimes G/\gamma_3(G)\otimes G/\gamma_3(G)|$. But $G/K$ is a non abelian $p$-group with $|\big(G/K\big)'|=p^2$ so $|\mathcal{M}^{(2)}\big(G/K\big)|\leq p^{\frac13(n-1)(n-2)(n-3)+1}$ by using Lemma \ref{cs}. Since $|K\otimes G/\gamma_3(G)\otimes G/\gamma_3(G)|\leq p^{(n-2)^2}$, the result  follows.
\end{proof}
\begin{lem}\label{ell}
Let $G$ be a $p$-group of order $p^n$ with $G'$ of order $p^2$. If $G/G'$ is not elementary, then $|\mathcal{M}^{(2)}(G)|=p^{\frac13n(n-1)(n-2)-2}$. 
\end{lem}
\begin{proof}
The result is obtained by a similar way used in the proof of Lemma \ref{nl} and Theorem \ref{2.8} and \ref{2.9}.
\end{proof}
The next lemma shows the same upper bound in Lemma \ref{lw} works when $Z(G)$ is of order $p$.

\begin{lem}
Let $G$ be a $p$-group of order $p^n$ with $G'$ of order $p^2$. If $|Z(G)|=p$ then $|\mathcal{M}^{(2)}(G)|\leq p^{\frac13n(n-1)(n-2)-2}.$
\end{lem}
\begin{proof}
By using Lemma \ref{p1} (a), when $B=Z(G)$ and Theorems \ref{2.8} and \ref{2.9}, the result follows. 
\end{proof}

The last case is the one for which $G'=Z(G)=\mathbb{Z}_p\oplus \mathbb{Z}_p$.

\begin{thm} There is no finite
 $p$-group of order $p^n$ with $G'=Z(G)=\mathbb{Z}_p\oplus \mathbb{Z}_p$ such that $|\mathcal{M}^{(2)}(G)|=p^{\frac13n(n-1)(n-2)}.$
\end{thm}

\begin{proof}
By contrary, let there be a finite $p$-group $G$ of order $p^n$ such that $|\mathcal{M}^{(2)}(G)|=p^{\frac13n(n-1)(n-2)}$ and  $G'=Z(G)=\mathbb{Z}_p\oplus \mathbb{Z}_p$.   Let $K$ be a central subgroup of order $p$ in $G'$ and using  Lemma \ref{p1}(a) we have
$|\mathcal{M}^{(2)}(G)|\leq |\mathcal{M}^{(2)}\big(G/K\big)||K\otimes G/G'\otimes G/G')|$. 
Now Theorems \ref{2.8} and \ref{2.9} shows that $|\mathcal{M}^{(2)}\big(G/K\big)|\leq p^{\frac13(n-1)(n-2)(n-3)+3} $ and also $G/G'$ is elementary abelian by using Lemma \ref{ell}. Hence $p^{\frac13n(n-1)(n-2)}=|\mathcal{M}^{(2)}(G)|\leq p^{\frac13(n-1)(n-2)(n-3)+3}p^{(n-2)^2}$. Hence $n\leq 5$. Since $n\neq 4$, we have $n=5$. Now \cite[page 345]{ni0} shows $G\cong \mathbb{Z}_p^{(4)}\rtimes  \mathbb{Z}_p$
By a same way used in the proof \cite[Theorem 3.5]{ni3}, we have $|\mathcal{M}^{(2)}(G)|=p^{18}$, which is a contradiction. Hence the assumption is false and the result follows.
\end{proof}

Now we summarize the results as follows.

\begin{thm}
Let $G$ be a non abelian $p$-group of order $p^n$. Then
\begin{itemize}
\item[(i)] there is no group $G$ with $|\mathcal{M}^{(2)}(G)|=p^{\frac13n(n-1)(n-2)+2}$.
\item[(ii)] $|\mathcal{M}^{(2)}(G)|=p^{\frac13n(n-1)(n-2)+1}$ if and only if $p=2$ and $G\cong D_8 \times \mathbb{Z}_2^{(n-3)}.$
\item[(iii)] $|\mathcal{M}^{(2)}(G)|=p^{\frac13n(n-1)(n-2)}$ if and only if $G\cong H_m\times \mathbb{Z}_p^{(n-2m-1)}$, in where $H_m$ is an extraspecial $p$-groups of order $p^{2m+1}$ and $m\geq 2$ or $G\cong H_m\cdot\mathbb{Z}_{p^2}\times \mathbb{Z}_{p}^{(n-2m-2)}$. 
\end{itemize}
\end{thm}


\begin{thebibliography}{20}
\bibitem{ba} Baer, R.: Representations of groups as quotient groups, I, II, and III. Trans. Amer. Math. Soc.
58, 295�419 (1945)
\bibitem{fr}Fr\"{o}hlich, A.: Baer invariants of algebras. Trans. Amer. Math. Soc. 109, 221�244 (1962)
\bibitem{be}  Ya.G. Berkovich, On the order of the commutator subgroups and the Schur multiplier of a finite p-group,
 J. Algebra 144 (1991) 269-272.
\bibitem{ma}MacDonald, J.L.: Group derived functors. J. Algebra 10, 448�477 (1968)
\bibitem{el}J. Burns and G. Ellis, Inequalities for Baer invariants of finite groups, Can. Math. Bull.
41(4) (1998), 385�391.
\bibitem{el3} G. Ellis, On the Schur multiplier of p-groups, Comm. Algebra 9 (1999) 4173-4177.

\bibitem{el2} J. Burns, G. Ellis, On the nilpotent multipliers of a group, Math. Z. 226 (1997) 405�428.
\bibitem{hal} Hall, M., Senior, J.K.: The groups of order $2n$ $(n\leq 6)$, MacMillan, New York (1964)
\bibitem{lu} A. S.-T. Lue, The Ganea map for nilpotent groups, J. London Math. Soc. 14(1976), 309�312.
 \bibitem{r3} M.R.R. Moghaddam, The Baer invariant of a direct product, Arch. Math. 33 (1980) 504-511.
\bibitem{r1} M.R.R. Moghaddam, Some inequalities for the Baer invariant of a finite group, Bull. Iranian Math.
Soc. 9 (1981)  5-10.
\bibitem{ni}  P. Niroomand, On the order of Schur multiplier of non-abelian $p$-groups, J. Algebra 322 (2009) 4479-4482.
\bibitem{ni1} P. Niroomand, The Schur multiplier of $p$-groups with large derived subgroup, Archiv der Math. 95 (2010) 101-103.
 \bibitem{ni2}  P. Niroomand, 	Classifying $p$-groups by their Schur multipliers. Math. Rep. (Bucur.) 20(70), 3 (2018), 279-284.
   \bibitem{ni0} P. Niroomand, A note on the Schur multiplier of groups of prime power order. Ric. Mat. 61 (2012), no. 2, 341--346
    \bibitem{ni3}  P. Niroomand. Characterizing finite $p$-groups by their Schur multipliers, $t(G)=5$. Math. Rep. (Bucur.) 17(67) (2015), no. 2, 249--254.

    \bibitem{ni3}P. Niroomand and M. Parvizi, On the 2-nilpotent multiplier of finite $p$-groups. Glasg. Math. J. 57 (2015), no. 1, 201--210.
\bibitem{zh}X. Zhou, On the order of Schur multipliers of finite p-groups, Comm. Algebra
1 (1994) 1-8.
\end{thebibliography}
\end{document}